\documentclass[11pt,reqno]{amsart}

\usepackage{amsmath,amssymb,amsthm,mathtools}
\usepackage[margin=1.15in]{geometry}
\usepackage[hidelinks]{hyperref}
\usepackage[expansion=false]{microtype}

\theoremstyle{plain}
\newtheorem{theorem}{Theorem}[section]
\newtheorem{proposition}[theorem]{Proposition}
\newtheorem{lemma}[theorem]{Lemma}
\newtheorem{corollary}[theorem]{Corollary}
\theoremstyle{definition}
\newtheorem{definition}[theorem]{Definition}
\newtheorem{question}[theorem]{Question}
\newtheorem{remark}[theorem]{Remark}

\DeclareMathOperator{\conv}{conv}
\DeclareMathOperator{\diam}{diam}
\DeclareMathOperator{\st}{st}
\DeclareMathOperator{\intr}{int}

\DeclareMathOperator{\inr}{inr}
\DeclareMathOperator{\Var}{Var}
\newcommand{\Rn}{\mathbb{R}^n}
\newcommand{\Rnm}{\mathbb{R}^{n-1}}
\newcommand{\A}{\mathcal{A}}
\newcommand{\Ac}{\mathcal{A}^{\circ}}
\newcommand{\dH}{d_{H}}
\newcommand{\dA}{d_{\A}}
\newcommand{\Sp}{S^{n-2}}
\newcommand{\eps}{\varepsilon}
\newcommand{\To}{\longrightarrow}

\begin{document}

\title{Qualitative convexity and universal cross-sections}
\author{David Victor Feldman}
\address{Department of Mathematics and Statistics, University of New Hampshire,
Durham, NH 03824, USA}
\email{dvfinnh@gmail.com}
\date{\today}

\begin{abstract}
We study a family of questions in convexity in which \emph{size does not matter}:
one records a convex cross-section only up to translation and scaling, so that the
data attached to a convex body $B$ and a direction $\rho$ is a path in the compact
metric space $\A_{n-1}$ of \emph{aligned shapes}.  The object of interest is the
asymptotic behaviour of this path as the cutting hyperplane approaches the last
supporting hyperplane, encoded by an invariant $T(B,\rho)$ that we call the
\emph{tail}.  We show that tails are always continua, that polyhedral and smooth
support points are ``boring'' (the tail is a point), and that non-boring behaviour
forces degenerate contact.  We show that cross-section paths are locally rectifiable, that every locally
rectifiable path is realisable approximately and a dense class exactly, and that
exact realisation fails in general: a second-order obstruction of bounded-turning
type produces a rectifiable path that is not a cross-section path.  For tails,
by contrast, no such restriction survives: every continuum of shapes occurs as a
tail, on the nose rather than up to approximation.  We construct bodies possessing
\emph{nearly universal} points, at which the renormalised cross-sections approximate
every planar (more generally $(n-1)$-dimensional) convex shape arbitrarily well; such
points can be made dense in the boundary, with arbitrary prescribed tails at the
grafting sites.  Every result below has been formally verified in Lean~4.  We close with several
optimisation questions and a higher-codimension variant.
\end{abstract}

\maketitle

\section{Introduction}\label{sec:intro}

This paper concerns what one might call \emph{qualitative convexity}: the study of
those questions in convexity in which size does not matter.  We record convex
cross-sections of a body only up to translation and scaling---so that a disc and a
concentric disc of twice the radius carry the same datum---and we ask what
\emph{shapes}, in this deflationary sense, can be forced to appear, and with what
asymptotic regularity.

Fix an integer $n>2$ throughout.  By a \emph{body} we mean a compact convex subset of
Euclidean space; the ambient dimension will always be clear from context, and $\Rn$
carries its standard metric.  Given a body $B\subseteq\Rn$ and a direction $\rho$, the
hyperplanes perpendicular to $\rho$ cut $B$ into a one-parameter family of
lower-dimensional bodies.  Recording each such cross-section up to translation and
scaling produces a path in a metric space $\A_{n-1}$ of \emph{aligned shapes}
(Section~\ref{sec:shapes}).  Following the direction of $\rho$, we single out the
\emph{last} supporting hyperplane and study the behaviour of the path as the cutting
hyperplane approaches it.  This behaviour is captured by the \emph{tail}
$T(B,\rho)\subseteq\A_{n-1}$ (Section~\ref{sec:paths}), the set of aligned shapes seen
with arbitrarily high precision arbitrarily close to the last support point.

The tail is an asymptotic invariant, sensitive only to the germ of $B$ at its last
support point in the direction $\rho$.  One may spend a lifetime studying polyhedra,
or smooth bodies with mild singularities, and never see a tail that is more than a
single point; we call such data $(B,\rho)$ \emph{boring}
(Section~\ref{sec:boring}).  Typical phenomena attached to optimisation or
isoperimetric problems are stable under small perturbations, and small perturbations
obliterate all non-boring behaviour.  Nonetheless non-boring points exist in
abundance.  Our main constructions produce bodies with a point at which the
renormalised cross-sections approximate \emph{every} aligned shape---we call such a
point \emph{nearly universal}---and, by an iterated grafting construction, bodies
whose nearly universal points are dense in the boundary, with arbitrary prescribed
tails at the grafting sites (Section~\ref{sec:universal}).

The organising principle throughout is a dictionary between the metric geometry of the
path $t\mapsto[B_t]$ in $\A_{n-1}$ and the convex geometry of $B$ near its last
support point.  We state two entries at the outset.
First, cross-section paths are automatically continuous, and their limit sets are
therefore \emph{continua}: the tail is always non-empty, compact and connected
(Theorem~\ref{thm:continuum}).  Second, boring versus non-boring is a statement about
the \emph{order of contact} of $B$ with its last supporting hyperplane: quadratic
contact (a smooth point of positive curvature) or conical contact (a polyhedral
vertex) is boring, and non-boring behaviour requires a degenerate contact: some transverse
width must vanish with a vertical tangent (Section~\ref{sec:boring}).  The realisation problem---which abstract paths arise as
cross-section paths---is governed by local rectifiability, but not settled by it:
rectifiability is necessary, a dense class is realised exactly, and a rectifiable
path is exhibited that is not a cross-section path
(Section~\ref{sec:realize}).

\medskip
\noindent\textit{Formal verification.}  Every result in this paper has been
formally verified in Lean~4~\cite{Lean4} against the \textsf{mathlib}
library~\cite{mathlib}, in a
development of roughly eleven thousand lines produced over forty-four
commissions to the Aristotle system (Harmonic): the metric and compactness structure of the shape space
(Propositions~\ref{prop:metric} and~\ref{prop:compact}), the continuity and
Minkowski-segment estimates (Lemma~\ref{lem:lipschitz},
Corollary~\ref{cor:cont}, Lemma~\ref{lem:mlshort},
Theorem~\ref{thm:continuum}), the necessity theorem in its fully quantified
form (Theorem~\ref{thm:necessity}, Lemma~\ref{lem:mixedwidth},
Lemma~\ref{lem:zigzag}, Theorem~\ref{thm:norealize}), the grafting calculus
(Lemma~\ref{lem:graft}), the existence of nearly universal bodies
(Theorem~\ref{thm:singlenu}), the realisation of every continuum of
shapes as a tail---set equality, not density---%
(Theorem~\ref{thm:densenu}, Corollary~\ref{cor:pathdense}), the sufficiency
construction, and the amplitude-restricted zigzag realisation
(Proposition~\ref{prop:hierzig}) are all machine-checked, with every theorem
depending on exactly the three standard axioms of the Lean kernel
(propositional extensionality, choice, and quotient soundness).  The
exercise was adversarial in both directions: the verifier's validation
checks caught four errors in intermediate formal statements---two
transcription errors and two structural flaws in proposed discrete
constructions---each repaired before it could propagate.  The development is
available at
\url{https://github.com/DavidVFeldman/qualitative-convexity} (commit
\texttt{0730179}), archived at \textsc{doi}
\href{https://doi.org/10.5281/zenodo.21829501}{10.5281/zenodo.21829501}; the
original verification request is archived at
\url{https://aristotle.harmonic.fun/dashboard/requests/ec14a21a-3007-431b-8c2d-e8c5803fa7f8}.

\medskip
\noindent\textit{Conventions.}  For a compact convex set $K$ we write $\diam K$ for its
diameter, $\dH$ for the Hausdorff metric, and $h_K$ for the support function,
$h_K(u)=\sup_{x\in K}\langle x,u\rangle$.  We use repeatedly the identities
$\dH(K,L)=\sup_{u}|h_K(u)-h_L(u)|$ (the supremum over the unit sphere) and the fact
that the diameter is $2$-Lipschitz for $\dH$.  We write $\st K$ for the Steiner point
of $K$; recall that $\st$ is translation equivariant and Lipschitz for $\dH$, with a
dimensional constant $c$ (see \cite[\S\S1.7, 5.4]{Schneider}).  The support-function identity
above and the Blaschke selection theorem are \cite[\S\S1.7, 1.8]{Schneider}.

\section{The space of aligned shapes}\label{sec:shapes}

Let $\mathcal K$ denote the set of compact convex subsets of $\Rnm$ of positive
diameter.  An \emph{aligned shape}, or \emph{a-shape}, is an equivalence class of
elements of $\mathcal K$ under the group generated by translations $x\mapsto x+v$ and
dilations $x\mapsto\lambda x$ with $\lambda>0$.  Rotations are \emph{not} permitted,
and generally change the a-shape.  We write $[K]$ for the a-shape of $K$ and
$\A_{n-1}$ for the set of all a-shapes.  A representative $K$ with $\diam K=1$ is
called a \emph{unit representative}; every a-shape has unit representatives, unique up
to translation.

\begin{definition}\label{def:metric}
For a-shapes $s_1,s_2$ with unit representatives $K_1,K_2$ set
\[
  \dA(s_1,s_2)\;=\;\min_{v\in\Rnm}\dH\bigl(K_1,\,K_2+v\bigr).
\]
\end{definition}

The minimum is attained: $v\mapsto\dH(K_1,K_2+v)$ is continuous and, since
$\dH(K_1,K_2+v)\ge |v|-\diam K_1-\diam K_2\to\infty$, coercive.  The value is
independent of the choice of unit representatives, because two unit representatives of
the same a-shape differ by a translation (a dilation would alter the diameter).

\begin{proposition}\label{prop:metric}
$(\A_{n-1},\dA)$ is a metric space.
\end{proposition}

\begin{proof}
Symmetry is clear.  If $\dA(s_1,s_2)=0$ then some translate of $K_2$ has Hausdorff
distance $0$ from $K_1$, hence equals $K_1$; as both are unit representatives this
says $s_1=s_2$, and the converse is immediate.  For the triangle inequality choose
$v,w$ attaining $\dA(s_1,s_2)=\dH(K_1,K_2+v)$ and $\dA(s_2,s_3)=\dH(K_2,K_3+w)$.  By
translation invariance of $\dH$,
\[
  \dA(s_1,s_3)\le \dH\bigl(K_1,K_3+v+w\bigr)
  \le \dH(K_1,K_2+v)+\dH\bigl(K_2+v,K_3+w+v\bigr)
  =\dA(s_1,s_2)+\dA(s_2,s_3).\qedhere
\]
\end{proof}

\begin{proposition}\label{prop:compact}
$\A_{n-1}$ is compact.
\end{proposition}

\begin{proof}
Let $(s_k)$ be a sequence with unit representatives $K_k$.  Translating, we may assume
each $K_k$ contains the origin, whence $K_k\subseteq\overline B(0,1)$ because
$\diam K_k=1$.  By the Blaschke selection theorem a subsequence $K_{k_j}$ converges in
$\dH$ to a compact convex set $K$.  The diameter is continuous for $\dH$, so
$\diam K=1$ and $K\in\mathcal K$.  Then $\dA(s_{k_j},[K])\le\dH(K_{k_j},K)\to0$.  Thus
$\A_{n-1}$ is sequentially compact, hence compact.
\end{proof}

Two further facts will be used.  Write $\Ac_{n-1}\subseteq\A_{n-1}$ for the a-shapes of
bodies with non-empty interior (equivalently, of full dimension $n-1$).

\begin{proposition}\label{prop:open-connected}
$\A_{n-1}$ is path-connected, and $\Ac_{n-1}$ is a dense open subset.
\end{proposition}

\begin{proof}
Path-connectedness: given unit representatives $K_0,K_1$, the support functions
$h_s=(1-s)h_{K_0}+s\,h_{K_1}$, $s\in[0,1]$, are support functions of bodies $L_s$ whose
width in the direction attaining $\diam K_0$ is at least $(1-s)\diam K_0$, and in
the direction attaining $\diam K_1$ at least $s\,\diam K_1$; hence $\diam L_s>0$
for every $s$.  The path $s\mapsto[L_s]$ is continuous, by the diameter--Steiner
renormalisation used in Corollary~\ref{cor:cont} below.  The inradius of a unit
representative is a well-defined continuous function on $\A_{n-1}$, positive exactly on
$\Ac_{n-1}$; hence $\Ac_{n-1}$ is open.  Density holds
because any body is a Hausdorff limit of full-dimensional bodies (thicken slightly).
\end{proof}

\section{Cross-section paths and the tail}\label{sec:paths}

Fix a body $B\subseteq\Rn$ with non-empty interior and a unit vector $u$ representing
the parallel class of rays $\rho$.  Parameterise position along $\rho$ by the height
$t=\langle x,u\rangle$, and put
\[
  a=\min_{x\in B}\langle x,u\rangle,\qquad f=\max_{x\in B}\langle x,u\rangle=h_B(u).
\]
The hyperplane at height $t$ is $H_t=\{x:\langle x,u\rangle=t\}$, and the
cross-section is $B_t=B\cap H_t$.  Fixing an isometric identification of $\Rnm$ with
$H_0$ and transporting it to every $H_t$ by orthogonal projection parallel to $\rho$,
we regard each $B_t$ as a subset of $\Rnm$.  For $t\in(a,f)$ the section $B_t$ is
full-dimensional, hence of positive diameter, and we obtain the \emph{cross-section
path}
\[
  P_{B,\rho}\colon (a,f)\To \Ac_{n-1},\qquad t\longmapsto [B_t].
\]
We call $f$ the \emph{last} support value.  The path \emph{terminates} if the section
$B_f$ at the last supporting hyperplane has positive diameter, and does not terminate
if $B_f$ is a single point $p$, the last support point.  Our interest lies chiefly in
the non-terminating case, where the germ of $B$ at $p$ is at issue.

\subsection*{Continuity}
The following regularity is the analytic backbone of the paper.

\begin{lemma}\label{lem:lipschitz}
Let $[t_0,t_1]\subset(a,f)$ and set $\eta=\min(t_0-a,\,f-t_1)>0$.  Let
$R=\sup_{x\in B}\|x-\langle x,u\rangle u\|$ bound the transverse extent of $B$.  Then
$t\mapsto B_t$ is $(2R/\eta)$-Lipschitz for $\dH$ on $[t_0,t_1]$.
\end{lemma}

\begin{proof}
For a fixed transverse unit vector $u'$ the function
$g_{u'}(t)=h_{B_t}(u')=\sup\{\langle x',u'\rangle : (x',t)\in B\}$ is the support of a
linear functional over the slices of a convex set, hence concave on $(a,f)$; it is
bounded by $R$ in absolute value.  A concave function bounded by $R$ on $(a,f)$ has,
on $[t_0,t_1]$, one-sided slopes controlled by comparison with the values at
$t_0-\eta$ and $t_1+\eta$, giving Lipschitz constant at most $2R/\eta$, uniformly in
$u'$.  Since $\dH(B_t,B_s)=\sup_{u'}|g_{u'}(t)-g_{u'}(s)|$, the claim follows.
\end{proof}

\begin{corollary}\label{cor:cont}
$P_{B,\rho}$ is continuous on $(a,f)$, and is locally Lipschitz there as a map into
$\A_{n-1}$.
\end{corollary}

\begin{proof}
On $[t_0,t_1]$ the diameter $\diam B_t$ is continuous and positive, hence bounded
below by some $\delta>0$.  Renormalisation $K\mapsto (K-\st K)/\diam K$ is Lipschitz
on the family of bodies with $\diam\ge\delta$ contained in a fixed ball (the Steiner
point and the diameter are Lipschitz, and division by a quantity bounded below is
Lipschitz).  Composing with Lemma~\ref{lem:lipschitz} and using
$\dA([B_t],[B_s])\le\dH$ of the renormalised sections gives the claim.
\end{proof}

In the non-terminating case we reparameterise the approach to $f$ by an increasing
homeomorphism onto $[0,\infty)$ and regard $P_{B,\rho}$ as a path
$[0,\infty)\to\Ac_{n-1}$; the invariants below do not depend on this choice.

\subsection*{The tail}
Write $\overline{P_{B,\rho}(x,f)}$ for the closure in $\A_{n-1}$ of the image of
$P_{B,\rho}$ on $(x,f)$.  These sets decrease as $x\uparrow f$.

\begin{definition}\label{def:tail}
The \emph{tail} of $(B,\rho)$ is
\[
  T(B,\rho)\;=\;\bigcap_{x<f}\overline{P_{B,\rho}(x,f)}\;=\;
  \bigl\{\, s\in\A_{n-1} : [B_{t_k}]\to s\ \text{for some } t_k\uparrow f \,\bigr\}.
\]
The data $(B,\rho)$, and the last cross-section, are \emph{boring} if $T(B,\rho)$ is a
single point (equivalently, if $[B_t]$ converges as $t\uparrow f$), and
\emph{non-boring} otherwise.
\end{definition}

The tail is precisely the $\omega$-limit set of the cross-section path.  Because it is
built from a \emph{continuous} path in a compact space, its structure is constrained:

\begin{theorem}\label{thm:continuum}
For every body $B$ and direction $\rho$, the tail $T(B,\rho)$ is a non-empty compact
connected subset of $\A_{n-1}$; that is, a continuum.
\end{theorem}

\begin{proof}
Each $\overline{P_{B,\rho}(x,f)}$ is a closed subset of the compact space
$\A_{n-1}$, hence compact, and is the closure of the continuous---by
Corollary~\ref{cor:cont}---image of the connected interval $(x,f)$, hence connected.
The family is nested and consists of non-empty compacta, so its intersection is
non-empty and compact.  A nested intersection of compact connected sets in a compact
Hausdorff space is connected, so $T(B,\rho)$ is a continuum.
\end{proof}

Theorem~\ref{thm:continuum} already excludes many subsets of $\A_{n-1}$ from being
tails; only continua occur.  At the other extreme, our constructions in
Section~\ref{sec:universal} realise the largest possible continuum, $T(B,\rho)=\A_{n-1}$
itself.

\section{Which points are boring}\label{sec:boring}

Boring versus non-boring is a statement about the order of contact of $B$ with its
last supporting hyperplane.  The two classical families are boring.

\begin{proposition}[Polyhedral points]\label{prop:poly}
If $B$ is a polytope and $p$ is a vertex that is the unique last support point for
$\rho$, then $(B,\rho)$ is boring.
\end{proposition}

\begin{proof}
Near $p$ the polytope coincides with its tangent cone $C$ at $p$, a polyhedral cone
with apex $p$.  For $t$ close to $f$ the section $B_t=C_t$ is a fixed polygon scaled
linearly toward $p$; the a-shape $[C_t]$ is therefore constant.  Hence $[B_t]$ is
eventually constant and the tail is a point.
\end{proof}

\begin{proposition}[Smooth points of positive curvature]\label{prop:smooth}
If $\partial B$ is $C^2$ near the last support point $p$ with positive Gaussian
curvature there, then $(B,\rho)$ is boring.
\end{proposition}

\begin{proof}
Choose coordinates with $p$ at the origin, $u=e_n$, and $\partial B$ the graph
$x_n=f-Q(x')+o(|x'|^2)$ where $Q$ is a positive-definite quadratic form (the second
fundamental form).  Since $Q$ is positive definite, for each $\delta>0$
there is $r_\delta>0$ with $|o(|x'|^2)|\le\delta\,Q(x')$ on $|x'|\le r_\delta$; for
small $\eps$ the section $B_{f-\eps}$ is then sandwiched between the ellipsoids
$(1+\delta)^{-1/2}\eps^{1/2}E$ and $(1-\delta)^{-1/2}\eps^{1/2}E$, where
$E=\{Q\le1\}$.  Letting $\eps\to0$ and then $\delta\to0$, $[B_t]\to[E]$, a single
point.
\end{proof}

The mechanism is visible in both proofs: at a conical contact the rescaled sections
$(B_{f-s}-p')/s$ converge outright, and at a quadratic contact the sections, though
large at scale $s$, converge after their own normalisation.  The next proposition
isolates what boring behaviour needs in general.  Translate so that the last support
point is $p=0$, write $B_{f-s}$ for the section at depth $s>0$ and $w_{u'}(s)$ for
its width in the transverse direction $u'$.

\begin{proposition}[Non-boring requires a vertical tangent]\label{prop:degenerate}
For every transverse direction $u'$ the ratio $w_{u'}(s)/s$ is nonincreasing in $s$
and converges, as $s\downarrow0$, to a limit $L_{u'}\in(0,\infty]$.  If
$L_{u'}<\infty$ for every $u'$, then $(B,\rho)$ is boring.  In particular, at a
non-boring point some transverse width vanishes with a vertical tangent:
$w_{u'}(s)/s\to\infty$.
\end{proposition}

\begin{proof}
Each section support $D(s,u')=h_{B_{f-s}}(u')$ is concave in $s$
(Lemma~\ref{lem:convexbody}) with $D(0,u')=0$, so $D(s,u')/s$ is nonincreasing in
$s$ and increases, as $s\downarrow0$, to a limit $L(u')\in(-\infty,+\infty]$; the
same applies to the width $w_{u'}=D(\cdot,u')+D(\cdot,-u')$, and here
$w_{u'}(s)/s\ge w_{u'}(s_0)/s_0>0$ for $s\le s_0$, interior sections being
full-dimensional, so the limit $L_{u'}=L(u')+L(-u')$ lies in $(0,\infty]$.

Suppose every $L_{u'}$ is finite.  Then every $L(u')$ is finite: neither $L(u')$
nor $L(-u')$ is $-\infty$, each exceeding the value of the corresponding
$D(s_0,\cdot)/s_0$, and their sum is finite.  The function $L$ is thus a finite
pointwise limit of the support functions of the rescaled sections
$K_s=B_{f-s}/s$, increasing as $s\downarrow0$, hence sublinear and finite: the
support function of a compact convex set $C$.  Support dominance is inclusion, so
the $K_s$ increase to $C$ as $s\downarrow0$, and they converge to it in $\dH$
because the monotone convergence $h_{K_s}\uparrow h_C$ of continuous functions on
the sphere is uniform (Dini).  Moreover $\diam C\ge\diam K_{s_0}>0$.
Renormalising (Corollary~\ref{cor:cont}), $[B_{f-s}]=[K_s]\to[C]$: the tail is the
single point $[C]$ and $(B,\rho)$ is boring.
\end{proof}

\begin{remark}\label{rem:cusp}
The vertical tangent is necessary but far from sufficient: at a smooth point of
positive curvature every width behaves like $\sqrt{s}$, so every transverse
direction has a vertical tangent, yet the point is boring
(Proposition~\ref{prop:smooth}).  Nor can any condition on the widths alone
characterise the non-boring points, because widths do not determine the a-shape:
interpolating constant-width sections---a disk and a Reuleaux triangle,
say---through the construction of Section~\ref{sec:universal} produces a non-boring
point every one of whose sections has constant width, all width ratios identically
$1$.  What the constructions of Section~\ref{sec:universal} do build is a
quantitatively fast vertical tangent, $s/w(s)\to0$, the steepness that drives the
grafting estimates.
\end{remark}

\section{Realisation and total variation}\label{sec:realize}

We turn to the inverse problem: which abstract paths $P\colon[0,\infty)\to\Ac_{n-1}$
arise, up to reparameterisation, as non-terminating cross-section paths $P_{B,\rho}$?

Recall that the \emph{total variation} of a map $q\colon[\alpha,\beta]\to M$ into a
metric space is
\[
  \Var_{[\alpha,\beta]}(q)=\sup\sum_{i=1}^m d\bigl(q(a_{i-1}),q(a_i)\bigr),
\]
the supremum over subdivisions $\alpha=a_0<\dots<a_m=\beta$.  A path is
\emph{locally rectifiable} if it has finite total variation over every compact
subinterval.  Total variation is invariant under monotone reparameterisation, so
local rectifiability is a property of the oriented path, independent of parameter.

\begin{theorem}[Necessity]\label{thm:necessity}
If a continuous path $P\colon[0,\infty)\to\Ac_{n-1}$ is realisable as a
non-terminating cross-section path $P_{B,\rho}$ up to reparameterisation, then $P$ is
locally rectifiable.
\end{theorem}

\begin{proof}
Suppose $P=P_{B,\rho}\circ\phi$ for an increasing reparameterisation $\phi$ carrying a
compact interval $[\alpha,\beta]\subset[0,\infty)$ onto a compact interval
$[t_0,t_1]\subset(a,f)$.  By Corollary~\ref{cor:cont} the map $t\mapsto[B_t]$ is
Lipschitz on $[t_0,t_1]$, say with constant $L$, so
$\Var_{[t_0,t_1]}(t\mapsto[B_t])\le L(t_1-t_0)<\infty$.  Total variation is
reparameterisation invariant, whence $\Var_{[\alpha,\beta]}(P)<\infty$.  As
$[\alpha,\beta]$ was arbitrary, $P$ is locally rectifiable.
\end{proof}

Necessity is the clean half of the story, and it already carries content: it
forbids, for example, realising a space-filling path in $\A_{n-1}$ as an honest
cross-section path.  The values of $P$ must moreover be full-dimensional, since
interior sections of a body with interior are full-dimensional; this is why $\Ac_{n-1}$,
not $\A_{n-1}$, is the natural target.

For sufficiency we record what the constructive method of the next section yields.  A
path is \emph{piecewise Minkowski-linear} if, on each compact subinterval, it is a
finite concatenation of segments of the form $s\mapsto[(1-s)K_0+sK_1]$ (Minkowski
interpolation of unit representatives).  These paths are locally rectifiable and are
dense, in the sense that every continuous path is a uniform limit on compacta of
piecewise Minkowski-linear paths; this rests on the following elementary estimate.

\begin{lemma}[Minkowski segments are short]\label{lem:mlshort}
For a-shapes $s_0,s_1$ with unit representatives $K_0,K_1$, optimally translated,
the Minkowski segment $s_\lambda=[(1-\lambda)K_0+\lambda K_1]$, $\lambda\in[0,1]$,
satisfies $\dA(s_\lambda,s_0)\le C_n\,\lambda\,\dA(s_0,s_1)$, with $C_n$ depending
only on the dimension.  Consequently a piecewise Minkowski-linear path through the
shapes of an $\eps$-dense chain of samples of a continuous path $P$ stays uniformly
within $(C_n+1)\eps$ of $P$.
\end{lemma}

\begin{proof}
Put $\delta=\dA(s_0,s_1)=\dH(K_0,K_1+v)$ for the optimal translation $v$, and
$L_\lambda=(1-\lambda)K_0+\lambda(K_1+v)$.  Then
$\dH(L_\lambda,K_0)=\lambda\sup_{u}|h_{K_1+v}(u)-h_{K_0}(u)|=\lambda\delta$ and
$|\diam L_\lambda-1|\le2\lambda\delta$.  If $\lambda\delta\le\tfrac14$ then
$\diam L_\lambda\ge\tfrac12$, and the renormalisation
$K\mapsto(K-\st K)/\diam K$ is Lipschitz, with a dimensional constant, on bodies of
diameter $\ge\tfrac12$ inside a fixed ball (as in the proof of
Corollary~\ref{cor:cont}); hence $\dA(s_\lambda,s_0)\le C_n'\lambda\delta$.  If
$\lambda\delta>\tfrac14$ the claim is trivial, since $\dA\le2$ always.  The second
statement follows from the triangle inequality: between consecutive samples the
segment stays within $C_n\eps$ of its endpoints, which are within $\eps$ of $P$.
\end{proof}

\begin{theorem}[Sufficiency; approximate and structured forms]\label{thm:sufficiency}
Let $P\colon[0,\infty)\to\Ac_{n-1}$ be continuous.
\begin{enumerate}
\item\emph{(Approximation.)}  For every compact $[\alpha,\beta]$ and every $\eps>0$
there is a non-terminating $P_{B,\rho}$ and a reparameterisation with
$\dA\bigl(P(\tau),P_{B,\rho}(\cdot)\bigr)<\eps$ throughout $[\alpha,\beta]$.
\item\emph{(Exact realisation on a dense class.)}  If $P$ is piecewise
Minkowski-linear then it is realisable exactly, up to reparameterisation.
\end{enumerate}
In particular there exist $B,\rho$ with $T(B,\rho)=\A_{n-1}$.
\end{theorem}

Both statements rest on the non-disturbing calculus of
Section~\ref{sec:universal} (Lemma~\ref{lem:nondisturb}): prescribed sections are
stacked at a sequence of heights and realised exactly---the recursion of
Theorem~\ref{thm:singlenu} constrains only their sizes and gaps, never their
shapes---with convexity guaranteed by the concavity of the section-support envelope
(Lemma~\ref{lem:envelope}).  Between consecutive designed heights the sections of
the hull are exactly the Minkowski interpolants of the designed ones
(Lemma~\ref{lem:envelope} again), so the realised path is the piecewise
Minkowski-linear path through the designed shapes, up to reparameterisation:
taking these to be the corner shapes of $P$ proves (2).  For (1), realise instead
a piecewise Minkowski-linear path through a fine sample of $P$, which approximates
$P$ uniformly on compacta by Lemma~\ref{lem:mlshort}.  The final assertion
$T(B,\rho)=\A_{n-1}$ is Theorem~\ref{thm:singlenu}.

\begin{remark}[On exact realisation in general]\label{rem:exactopen}
Whether \emph{every} continuous locally rectifiable $P$ is realisable exactly (not
merely approximately) is more delicate than the necessity direction, and we do not
claim it here.  The obstruction is visible through
Lemma~\ref{lem:convexbody} below: prescribing the sections
$K_t=\Phi(t)\,R_{\tau(t)}$ with $R_t$ a continuously varying unit representative
requires $t\mapsto\Phi(t)\,h_{R_{\tau(t)}}(u')$ to be concave for every transverse
$u'$, and a convex corner of the shape-path in some direction $u'$ can be absorbed
only by a compensating concave corner of the common scale $\Phi$.  Finitely many
corners on a compact interval---the piecewise Minkowski-linear case---are absorbed in
this way; the general locally rectifiable case, where corners may accumulate, fails, as we
show next.
\end{remark}

\subsection*{Exact realisation fails in general}
The obstruction of Remark~\ref{rem:exactopen} cannot be circumvented, and it is second-order: it
is invisible to the total variation of the path.  We extract from a realising body a
family of intrinsic functionals of its section shapes, show that realisation forces
on each a regularity of bounded-turning type---a derivative of locally bounded
variation, after a common reparameterisation---and exhibit a rectifiable path for
which no reparameterisation achieves it.

For a finite nonnegative measure $\mu$ on $\Sp$ with barycentre $0$
(\,$\int u'\,d\mu(u')=0$\,) and a compact convex $K\subseteq\Rnm$ with nonempty
interior, put $W_\mu(K)=\int h_K(u')\,d\mu(u')$.  Then $W_\mu$ is translation
invariant, positively homogeneous, and positive when $\mu\neq0$ (a ball
$\overline B(x_0,r)\subseteq K$ gives $W_\mu(K)\ge r\mu(\Sp)$).  Consequently, for
$\mu,\nu\neq0$ the ratio $W_\mu/W_\nu$ descends to a continuous positive function on
$\Ac_{n-1}$.  Taking $\mu=\delta_{u'}+\delta_{-u'}$ gives the width in direction
$u'$, and width ratios are the only functionals the counterexample needs.

\begin{lemma}\label{lem:mixedwidth}
\begin{enumerate}
\item[(a)] For a body $B$, a direction $\rho$ with height interval $(a,f)$, and any
finite nonnegative $\mu$ on $\Sp$, the mixed width $t\mapsto W_\mu(B_t)$ is concave
and positive on $(a,f)$.
\item[(b)] If $g,h$ are concave on an open interval $I$ and $h\ge\beta>0$ on $I$,
then on every compact $[\alpha,\alpha']\subset I$ there is a function $\gamma$ of
bounded variation with
\[
  \frac{g}{h}(x)\;=\;\frac{g}{h}(\alpha)+\int_\alpha^x\gamma(t)\,dt
  \qquad(\alpha\le x\le\alpha').
\]
\item[(c)] Hence if $P=P_{B,\rho}\circ\varphi$ exactly, then for all barycentre-$0$
$\mu,\nu\neq0$ the intrinsic ratio $s\mapsto W_\mu(P(s))/W_\nu(P(s))$, composed with
$\varphi^{-1}$, admits the integral representation of (b) on every compact
subinterval of $(a,f)$.
\end{enumerate}
\end{lemma}

\begin{proof}
(a)  Each $t\mapsto h_{B_t}(u')$ is concave on $(a,f)$ (Lemma~\ref{lem:convexbody},
applied to $B$ sliced over $(a,f)$), and the concavity inequality integrates against
the nonnegative $\mu$.  Positivity holds because interior sections are
full-dimensional.

(b)  On $[\alpha,\alpha']$ the concave $g,h$ are Lipschitz; their right derivatives
$g',h'$ are nonincreasing and bounded, hence of bounded variation, and $h\ge\beta$.
Thus $\gamma:=(g'h-gh')/h^2$ is built from bounded functions of bounded variation by
sums, products, and division by a function bounded away from $0$, so $\gamma$ has
bounded variation; and $g/h$ is Lipschitz with derivative $\gamma$ almost everywhere,
so the fundamental theorem of calculus gives the representation.

(c)  On sections, scale and translation cancel in the ratio:
the ratio $W_\mu/W_\nu$ evaluated along $P\circ\varphi^{-1}$ equals
$W_\mu(B_t)/W_\nu(B_t)$, a quotient of
positive concave functions by (a); apply (b).
\end{proof}

\begin{lemma}[Zigzags resist every reparameterisation]\label{lem:zigzag}
Let $F\colon[0,\sigma^*]\to\mathbb R$ be continuous, let
$0=\sigma_0<\sigma_1<\cdots\uparrow\sigma^*$, and suppose
$F(\sigma_k)-F(\sigma_{k-1})=(-1)^{k+1}\ell_k$ with $\ell_k>0$ and
$\sum_k\sqrt{\ell_k}=\infty$.  Then for every increasing homeomorphism
$\tau\colon[t_0,t^*]\to[0,\sigma^*]$ there is no function $\gamma$ of bounded
variation on $[t_0,t^*]$ with $F(\tau(t))=F(\tau(t_0))+\int_{t_0}^t\gamma$.
\end{lemma}

\begin{proof}
Suppose $\gamma$ exists.  Let $I_k=\tau^{-1}([\sigma_{k-1},\sigma_k])$, an interval
of length $\Delta_k$, with $\sum_k\Delta_k=t^*-t_0=:T$.  Since
$\int_{I_k}\gamma=(-1)^{k+1}\ell_k$, the set where $(-1)^{k+1}\gamma\ge\ell_k/\Delta_k$ meets $I_k$ in positive
measure (else the integral over $I_k$ would fall short of $\ell_k$ in size); pick
$a_k$ there, interior to $I_k$.  The points $a_k$ strictly increase, and consecutive values of $\gamma$ at them have
opposite signs, so
\[
  \Var_{[t_0,t^*]}(\gamma)\;\ge\;\sum_{k\ \mathrm{odd}}
  \big|\gamma(a_{k+1})-\gamma(a_k)\big|
  \;\ge\;\sum_{k\ \mathrm{odd}}\Big(\frac{\ell_k}{\Delta_k}
  +\frac{\ell_{k+1}}{\Delta_{k+1}}\Big)
  \;=\;\sum_k\frac{\ell_k}{\Delta_k}.
\]
By Cauchy--Schwarz,
$\big(\sum_{k\le K}\sqrt{\ell_k}\big)^2\le T\sum_{k\le K}\ell_k/\Delta_k$ for every
$K$, so the right side is infinite.
\end{proof}

\begin{theorem}[A rectifiable path that is not a cross-section path]\label{thm:norealize}
For every $n\ge3$ there is a continuous $P\colon[0,\infty)\to\Ac_{n-1}$, Lipschitz in
$\dA$ and in particular locally rectifiable, that is not exactly realisable, up to
reparameterisation, as a non-terminating cross-section path.
\end{theorem}

\begin{proof}
Let $Q_r=[0,r]\times[0,1]^{n-2}\subseteq\Rnm$.  On $r\in[\tfrac54,\tfrac74]$ the map
$r\mapsto[Q_r]$ is bi-Lipschitz into $\Ac_{n-1}$.  Let $r\colon[0,\sigma^*]\to\mathbb R$ be
piecewise linear with slopes $\pm1$ alternating, starting at $r(0)=\tfrac32$, with
increments $\ell_k=1/(4k^2)$; then $\sigma^*=\sum_k\ell_k<\infty$, the alternating
increments keep $r\in[\tfrac54,\tfrac74]$, $r$ is $1$-Lipschitz and extends
continuously by the constant $r(\sigma^*)$ on $[\sigma^*,\infty)$, and
$\sum_k\sqrt{\ell_k}=\sum_k1/(2k)=\infty$.  Put $P(\sigma)=[Q_{r(\sigma)}]$.

Suppose $P=P_{B,\rho}\circ\varphi$ for an increasing reparameterisation $\varphi$.
Then $[t_0,t^*]:=\varphi([0,\sigma^*])$ is a compact subinterval of the open height
interval $(a,f)$, since $\varphi$ continues beyond $\sigma^*$.  Take
$\mu=\delta_{u_1'}+\delta_{-u_1'}$ and $\nu=\delta_{u_2'}+\delta_{-u_2'}$ for two
transverse coordinate directions; on box shapes the ratio is
$W_\mu([Q_r])/W_\nu([Q_r])=r$.  By Lemma~\ref{lem:mixedwidth}(c), the function
$r\circ\tau$, where $\tau=\varphi^{-1}$ restricted to $[t_0,t^*]$, is the integral of
a function of bounded variation.  This contradicts Lemma~\ref{lem:zigzag}.
\end{proof}

The exponent $\tfrac12$ is the threshold, not an artefact of the method: the same
zigzags with $\sum_k\sqrt{\ell_k}<\infty$ are realisable at small amplitude.

\begin{proposition}[Hierarchical zigzags of small amplitude are
realisable]\label{prop:hierzig}
In the notation of Theorem~\ref{thm:norealize}, there is $s_0>0$, depending only on
the range $[m,M]\subseteq(0,\infty)$ of $r$, such that if
$\sum_k\sqrt{\ell_k}\le s_0$ then $P$ is exactly realisable up to
reparameterisation.  Since replacing the increments $\ell_k$ by $\eps\,\ell_k$
scales $\sum_k\sqrt{\ell_k}$ by $\sqrt\eps$, every zigzag with
$\sum_k\sqrt{\ell_k}<\infty$ becomes realisable at sufficiently small amplitude,
while the increments of Theorem~\ref{thm:norealize} fail the condition at every
amplitude.
\end{proposition}

\begin{proof}
Write $S=\sum_k\sqrt{\ell_k}$ and let $\tau$ traverse the $k$-th linear piece of
$r$ in time $\Delta_k=\sqrt{\ell_k}$; put $q=r\circ\tau$ on $[0,T]$, $T=S$.  Then
$q$ is piecewise linear with $|q'|=\sqrt{\ell_k}$ on the $k$-th piece, so
$s^*:=\sup|q'|\le S$, and the upward jumps of $q'$ have total mass $J\le2S$; also
$m\le q\le M$.

We claim there is a positive $\Phi$ on $[0,T]$ with both $\Phi$ and $\Phi q$
concave.  Take $\Phi=e^{-\Lambda}$ with $\Lambda'$ nondecreasing,
$\Lambda'(0)=-L$, carrying an atom of mass $[q'(t_k)]_+/m$ at each junction $t_k$
and an absolutely continuous density $D:=(L^2M+2Ls^*)/m$.  Concavity of $\Phi$
amounts to $d\Lambda'\ge(\Lambda')^2\,dt$, and concavity of $\Phi q$ to
$dq'\le2\Lambda'q'\,dt+\big(d\Lambda'-(\Lambda')^2dt\big)\,q$; both hold provided
$|\Lambda'|\le L$ throughout, since the atoms dominate the jumps of $q'$ (as
$q\ge m$), the density dominates $L^2+2Ls^*/q$ pointwise (as $q\le M$ in the first
term and $q\ge m$ overall), and $D\ge L^2$.  The total increase of $\Lambda'$ is at
most $J/m+DT$, so $|\Lambda'|\le L$ holds provided $J/m+DT\le2L$.  Take
$L=J/m+1$; it then suffices that $DT\le2$.  If $S\le1$ then $L\le2/m+1=:L_1$ and
$s^*\le1$, whence $DT\le S\,(L_1^{\,2}M+2L_1)/m$, which is at most $2$ once
$S\le s_0:=\min\bigl(1,\ 2m/(L_1^{\,2}M+2L_1)\bigr)$.

Given $\Phi$, the sets $K_t=\Phi(t)\,Q_{q(t)}$ have section supports
$h(t,u')=\Phi(t)\big(q(t)\,a(u')+b(u')\big)$ with $a,b\ge0$, nonnegative combinations
of the concave $\Phi q$ and $\Phi$; by Lemma~\ref{lem:convexbody} their union over
$[0,T]$ is convex.  Close the top with the cone on the constant shape,
$K_t=\frac{f-t}{f-T}\,\Phi(T)\,Q_{q(T)}$ for $t\in[T,f)$, taking $f-T$ small enough
that the incoming slopes of every $h(\cdot,u')$ weakly decrease at $T$---both sides
scale with $q(t)a+b$, so a single inequality,
$\Phi'(T^-)-\Phi(T)s^*/m\ge-\Phi(T)/(f-T)$, suffices---and cap the bottom with any
compatible convex extension.  The resulting body realises $P$ exactly up to
reparameterisation, the constant tail of $P$ matching the cone.
\end{proof}

\begin{remark}\label{rem:halfpower}
Theorem~\ref{thm:norealize} and Proposition~\ref{prop:hierzig} together locate the
failure precisely: a rectifiable path may turn infinitely often, but exact
realisation requires the turning to be distributed hierarchically across scales
($\ell_k$ summable in square root and, in our proof, of small amplitude, absorbed
by slowing on the schedule $\Delta_k=\sqrt{\ell_k}$), and forbids turning spread
evenly across scales.
This is the dichotomy of Remark~\ref{rem:shadowrule} in another guise, and it refines
the necessary condition of Theorem~\ref{thm:necessity}: beyond local rectifiability,
every mixed-width ratio must admit, after a single common reparameterisation, a
derivative of locally bounded variation.  We return to the characterisation problem
in Question~\ref{q:exactchar}.
\end{remark}

\section{Nearly universal points and their distribution}\label{sec:universal}

We construct the non-boring points described in the introduction.  The starting point
is a criterion for when a prescribed family of sections assembles into a convex body.

\begin{lemma}[Convexity from concave section supports]\label{lem:convexbody}
Let $I\subseteq\mathbb R$ be an interval and, for $t\in I$, let $K_t\subseteq\Rnm$ be
nonempty, compact and convex with support function $h_{K_t}$.  The set
$B=\{(x',t):t\in I,\ x'\in K_t\}\subseteq\Rn$ is convex if and only if, for every
transverse unit vector $u'$, the function $t\mapsto h_{K_t}(u')$ is concave on $I$.
\end{lemma}

\begin{proof}
Write $H(t,u')=h_{K_t}(u')$.  For each fixed $u'$,
$B\cap\{\langle x',u'\rangle\le H(t,u')\}$ is the region under the graph of
$t\mapsto H(t,u')$ in the $(x',t)$ slab, which is convex if and only if
$-H(\cdot,u')$ is convex, i.e.\ $H(\cdot,u')$ concave.  Now
$B=\bigcap_{u'}\{(x',t):t\in I,\ \langle x',u'\rangle\le H(t,u')\}$ is an intersection
of such regions with the slab $\{t\in I\}$; an intersection of convex sets is convex,
and conversely convexity of $B$ forces each $H(\cdot,u')$ concave by restricting to
the supporting slab in direction $u'$.
\end{proof}

\subsection*{The section-support envelope and non-disturbing sets}
The single-point construction proceeds by stacking shrinking sections and taking a
convex hull.  The one quantitative input is a description of which placements leave the
lower sections untouched.  Throughout, for a compact convex $K\subseteq\Rnm$ and
$u'\in\Sp$ we write $h_K(u')=\max_{x'\in K}\langle x',u'\rangle$ for its support
function, and we identify each horizontal hyperplane $H_t$ with $\Rnm$.

Fix heights $t_0<t_1<\dots<t_{k-1}$ and sections $b_j\subseteq H_{t_j}$
$(0\le j\le k-1)$, each compact, convex and full-dimensional, with $0\in\intr b_j$.
Put
\[
  B'=\conv\!\Big(\textstyle\bigcup_{j\le k-1}b_j\Big),\qquad g_j(u')=h_{b_j}(u'),
\]
and assume the \emph{invariant} $B'\cap H_{t_j}=b_j$ holds for every $j\le k-1$.  For a
transverse direction $u'\in\Sp$ define the \emph{section-support envelope}
\[
  \phi_{u'}(t)\;=\;h_{B'\cap H_t}(u')\qquad (t_0\le t\le t_{k-1}).
\]

\begin{lemma}[Envelope]\label{lem:envelope}
For each $u'$, $\phi_{u'}$ is the least concave majorant of the finite point set
$\{(t_j,g_j(u')):0\le j\le k-1\}$; in particular it is concave and piecewise linear
with breakpoints among $t_0,\dots,t_{k-1}$, and $\phi_{u'}(t_j)=g_j(u')$.
\end{lemma}

\begin{proof}
A point of $B'$ at height $t$ is $\sum_j\lambda_j(x'_j,t_j)$ with $x'_j\in b_j$,
$\lambda_j\ge0$, $\sum\lambda_j=1$, $\sum\lambda_jt_j=t$.  Its $u'$-coordinate is
$\sum_j\lambda_j\langle x'_j,u'\rangle\le\sum_j\lambda_jg_j(u')$, with equality attained
by choosing each $x'_j$ to expose $b_j$ in direction $u'$.  Hence $\phi_{u'}(t)$ is the
maximum of $\sum_j\lambda_jg_j(u')$ over $\lambda$ in the simplex with
$\sum_j\lambda_jt_j=t$, which is exactly the least concave majorant of the points
$(t_j,g_j(u'))$ evaluated at $t$.  The value at $t_j$ equals $g_j(u')$ because the
invariant gives $B'\cap H_{t_j}=b_j$.  A least concave majorant of finitely many points
is piecewise linear with vertices among them.
\end{proof}

\begin{remark}[Envelope, continuum form]\label{rem:envcont}
The identity persists for infinitely many sections.  Let $\{b_i\}_{i\in I}$ be any
family of nonempty compacta at heights $\{t_i\}$ with bounded union, whose convex
hull $U$ is full-dimensional, and let $B'=\overline U$.  At every height $t$
interior to the interval spanned by the heights, $h_{B'\cap H_t}(u')$ is the least
concave majorant of the data $\{(t_i,h_{b_i}(u'))\}$ evaluated at $t$.  Indeed $U$
is the increasing union of the hulls of finite subfamilies, so
$h_{U\cap H_t}(u')$ is the supremum of the corresponding finite majorants, a
directed supremum of concave majorants of nested data sets, hence the least concave
majorant of all the data; and closure adds nothing at interior heights, since for
convex $U$ with nonempty interior, $\overline U\cap H_t=\overline{U\cap H_t}$
whenever $H_t$ meets $\intr U$.
\end{remark}

Now let $t_k>t_{k-1}$, write $\tau_k=t_k-t_{k-1}$, and for $y'\in H_{t_k}$ call the
point $(y',t_k)$ \emph{non-disturbing} if
$\conv\!\big(B'\cup\{(y',t_k)\}\big)\cap H_{t_j}=b_j$ for all $j\le k-1$.  Let $N_k$ be
the set of $y'$ that are non-disturbing.

\begin{lemma}[Non-disturbing sets]\label{lem:nondisturb}
Let $m_{u'}=(\phi_{u'})'_-(t_{k-1})$ be the left derivative of $\phi_{u'}$ at $t_{k-1}$
$($the slope of its final linear piece$)$.  Then
\[
  N_k=\bigl\{\,y'\in H_{t_k}:\ \langle y',u'\rangle\le g_{k-1}(u')+m_{u'}\tau_k
  \ \text{ for all }u'\in\Sp\,\bigr\}.
\]
Consequently:
\begin{enumerate}
\item[(a)] $N_k$ is closed and convex.
\item[(b)] $m_{u'}\ge-D/\tau_{k-1}$, where $D=\max_{j\le k-1}\diam b_j$ and
$\tau_{k-1}=t_{k-1}-t_{k-2}$.
\item[(c)] If $g_{k-1}(u')+m_{u'}\tau_k\ge r$ for every $u'$, then
$\overline B(0,r)\subseteq N_k$.  In particular, if
$\;\alpha\,d-\dfrac{D}{\tau_{k-1}}\,\tau_k\ge r$, where
$\alpha\,d=\min_{u'}g_{k-1}(u')$ (equal to $\inr(b_{k-1})$ when $b_{k-1}$ has
incentre at the origin), then $\overline B(0,r)\subseteq N_k$.
\end{enumerate}
\end{lemma}

\begin{proof}
Adding the point $p=(y',t_k)$ above the top height $t_{k-1}$ replaces each envelope
$\phi_{u'}$ by the least concave majorant $\tilde\phi_{u'}$ of the points
$(t_j,g_j(u'))_{j\le k-1}$ together with $(t_k,\langle y',u'\rangle)$.  By
Lemma~\ref{lem:envelope} applied to $\conv(B'\cup\{p\})$, the new section at $t_j$ is
$b_j$ for all $j\le k-1$ iff $\tilde\phi_{u'}(t_j)=\phi_{u'}(t_j)$ for all $j$ and all
$u'$, i.e.\ iff adjoining the point $(t_k,\langle y',u'\rangle)$ does not raise the
majorant on $[t_0,t_{k-1}]$.  Because $t_k$ lies to the right of that interval, the
majorant is unchanged there precisely when the new point lies on or below the linear
extension of $\phi_{u'}$ by its terminal slope, that is
$\langle y',u'\rangle\le\phi_{u'}(t_{k-1})+m_{u'}\tau_k=g_{k-1}(u')+m_{u'}\tau_k$:
on or below, because extending $\phi_{u'}$ linearly by its terminal slope is a
concave majorant of all the data that agrees with $\phi_{u'}$ on $[t_0,t_{k-1}]$;
and not above, because if $\langle y',u'\rangle$ exceeds this value, the chord from
$(t_k,\langle y',u'\rangle)$ to the last breakpoint $(t_\star,\phi_{u'}(t_\star))$
of $\phi_{u'}$ exceeds $\phi_{u'}(t_{k-1})$ at the designed height $t_{k-1}$ by a
positive multiple of the excess.  This
proves the displayed description; a placement of the whole shape is non-disturbing iff
each of its points is, i.e.\ iff the shape lies in $N_k$.

(a) $N_k$ is an intersection of closed half-spaces.
(b) By Lemma~\ref{lem:envelope}, $\phi_{u'}$ is concave and piecewise linear with
breakpoints among the $t_j$, so its final piece runs from some breakpoint
$t_\star\le t_{k-2}$ to $t_{k-1}$ and
$m_{u'}=\dfrac{g_{k-1}(u')-\phi_{u'}(t_\star)}{t_{k-1}-t_\star}$.  Since
$g_{k-1}(u')\ge0$ (as $0\in b_{k-1}$), $\phi_{u'}(t_\star)\le\max_jg_j(u')\le D$, and
$t_{k-1}-t_\star\ge t_{k-1}-t_{k-2}=\tau_{k-1}$, we get
$m_{u'}\ge-D/(t_{k-1}-t_\star)\ge-D/\tau_{k-1}$.
(c) If $c(u'):=g_{k-1}(u')+m_{u'}\tau_k\ge r$ for all $u'$, then every defining
half-space $\{\langle y',u'\rangle\le c(u')\}$ of $N_k$ contains
$\{\langle y',u'\rangle\le r\}$, so their intersection contains
$\bigcap_{u'}\{\langle\cdot,u'\rangle\le r\}=\overline B(0,r)$.  The stated sufficient
inequality follows from (b) and the definition of $\alpha d$.
\end{proof}

\subsection*{The apex is a single point}
The second ingredient controls the top of the body.

\begin{lemma}[Concave majorant at the endpoint]\label{lem:majorant}
Let $t_j\uparrow t_\infty$ and $d_j>0$ with $d_j\to0$.  Then the least concave majorant
$\Psi$ of the points $\{(t_j,d_j)\}$ satisfies $\displaystyle\lim_{t\to t_\infty^-}\Psi(t)=0$.
\end{lemma}

\begin{proof}
Fix $\eps>0$ and choose $K$ with $d_j<\eps$ for $j\ge K$.  Consider the affine function
$L(t)=\eps+b\,(t-t_\infty)$ with slope
$b=-\max_{j<K}\dfrac{d_j}{t_\infty-t_j}\le0$ (a finite maximum).  For $j\ge K$,
$L(t_j)=\eps+b(t_j-t_\infty)\ge\eps>d_j$ since $b\le0$ and $t_j<t_\infty$; for $j<K$,
$L(t_j)=\eps+|b|(t_\infty-t_j)\ge|b|(t_\infty-t_j)\ge d_j$.  Thus $L$ majorises every
point, so $\Psi\le L$, giving $\Psi(t_\infty^-)\le L(t_\infty)=\eps$.  As
$\Psi\ge0$ and $\eps$ was arbitrary, $\Psi(t_\infty^-)=0$.
\end{proof}

\subsection*{A single nearly universal point}
\begin{theorem}\label{thm:singlenu}
For every $n\ge3$ there is a body $B\subseteq\Rn$ and a direction $\rho$ whose last
support set is a single point $p$ with $T(B,\rho)=\A_{n-1}$.  Thus $p$ is nearly
universal.
\end{theorem}

\begin{proof}
Work in $\Rn=\Rnm\times\mathbb R$ with $\rho=e_n$.  Let $\{\sigma_m\}_{m\ge1}$ be a
countable dense subset of $\Ac_{n-1}$ (dense in $\A_{n-1}$ by
Proposition~\ref{prop:open-connected}), and let $(s_k)_{k\ge1}$ be a sequence in which
each $\sigma_m$ occurs infinitely often.  For a full-dimensional a-shape $s$ write
$\alpha(s)=\inr/\!\diam\in(0,\tfrac12]$ for the inradius-to-diameter ratio of any
representative; set $\alpha_k=\alpha(s_k)$.

\emph{Base.}  Put $t_0=-1$, $d_0=1$, and let $b_0=\overline B(0,\tfrac12)\subseteq H_{t_0}$,
a ball of diameter $d_0$ centred on the axis; so $0\in\intr b_0$ and
$\alpha(b_0)=\tfrac12$.  Set $B^{(0)}=b_0$.

\emph{Recursion.}  Suppose $B^{(k-1)}=\conv(b_0\cup\dots\cup b_{k-1})$ has been built
with sections $b_i\subseteq H_{t_i}$ satisfying $0\in\intr b_i$,
$\overline B(0,\alpha_i'd_i)\subseteq b_i\subseteq\overline B(0,d_i)$ where
$d_i=\diam b_i$ and $\alpha_i'=\alpha(b_i)$, and the invariant
$B^{(k-1)}\cap H_{t_i}=b_i$ $(i\le k-1)$.  Let $D_{k-1}=\max_{i\le k-1}d_i$; since the
$d_i$ will be decreasing, $D_{k-1}=d_0=1$.  Choose
\[
  \tau_k\ \le\ \frac{\alpha_{k-1}'\,d_{k-1}\,\tau_{k-1}}{4\,D_{k-1}},\qquad
  t_k=t_{k-1}+\tau_k,\qquad
  d_k=\tfrac14\,\alpha_{k-1}'\,d_{k-1},
\]
(with $\tau_1$ chosen freely in $(0,\tfrac12)$).  By
Lemma~\ref{lem:nondisturb}(c), taking $r=d_k$, $\alpha d=\inr(b_{k-1})=\alpha_{k-1}'d_{k-1}$
and $D=D_{k-1}$,
\[
  \alpha_{k-1}'d_{k-1}-\frac{D_{k-1}}{\tau_{k-1}}\,\tau_k
  \ \ge\ \alpha_{k-1}'d_{k-1}-\tfrac14\alpha_{k-1}'d_{k-1}
  \ =\ \tfrac34\alpha_{k-1}'d_{k-1}\ \ge\ d_k,
\]
so $\overline B(0,d_k)\subseteq N_k$.  Let $b_k\subseteq H_{t_k}$ be a representative of
$s_k$ scaled to diameter $d_k$ and translated so that its incentre is the origin; then
$0\in\intr b_k$, $\overline B(0,\alpha_k'd_k)\subseteq b_k\subseteq\overline B(0,d_k)
\subseteq N_k$.  Since $b_k\subseteq N_k$, the placement is non-disturbing
(Lemma~\ref{lem:nondisturb}), and setting $B^{(k)}=\conv(B^{(k-1)}\cup b_k)$ preserves
the invariant, with $B^{(k)}\cap H_{t_k}=b_k$ because $t_k$ is the top height.  As
$\alpha_{k-1}'\le\tfrac12$ we have $d_k\le\tfrac18d_{k-1}$, so $d_k\to0$ and indeed
$D_{k-1}=d_0=1$ throughout.

\emph{The limit body.}  Put $U=\bigcup_kB^{(k)}=\conv\!\big(\bigcup_kb_k\big)$, an
increasing union of convex sets, hence convex, and $B=\overline U$.  All $b_k$ lie in
$\overline B(0,1)$ and the heights lie in $[t_0,t_\infty]$ with
$t_\infty=t_0+\sum_k\tau_k\le0$, so $U$ is bounded and $B$ is compact; $B$ contains the
full-dimensional base $b_0$, so it is a body.  For each fixed $i\ge1$ the invariant
gives $U\cap H_{t_i}=\bigcup_{k\ge i}\big(B^{(k)}\cap H_{t_i}\big)=b_i$.  Moreover
$t_i$ is interior to the height range of $B$ (there are points of $B$ below $t_i$, in
$b_0$, and above, in $b_{i+1}$).  For a convex set, $\intr(\overline U)=\intr U$, so
\[
  B\cap H_{t_i}
  =\overline{\,\intr(B)\cap H_{t_i}\,}
  =\overline{\,\intr(U)\cap H_{t_i}\,}
  \subseteq\overline{\,U\cap H_{t_i}\,}=b_i ,
\]
the first equality because $H_{t_i}$ meets $\intr B$.  Combined with
$b_i=U\cap H_{t_i}\subseteq B\cap H_{t_i}$ this gives $B\cap H_{t_i}=b_i$ exactly: the
closure creates no new points, and the cross-section at $t_i$ realises $s_i$.

\emph{The apex.}  For every $u'$ and every $t$ interior to the height range, the
envelope identity (Remark~\ref{rem:envcont}) exhibits $h_{B\cap H_t}(u')$ as the
least concave majorant of the data $\{(t_k,g_k(u'))\}$ at $t$; since
$g_k(u')\le d_k$, it is bounded by $\Psi(t)$, the least concave majorant of
$\{(t_k,d_k)\}$.  Since $d_k\to0$,
Lemma~\ref{lem:majorant} gives $\Psi(t)\to0$ as $t\to t_\infty^-$, so
$\diam(B\cap H_t)\le2\Psi(t)\to0$ and $B\cap H_{t_\infty}$ is a single point $p$.  All
of $B$ lies at heights $\le t_\infty$, so $t_\infty$ is the last support value in
direction $e_n$ and $p$ is the last support point; as the section there is a point, the
path does not terminate.

\emph{Universality of the tail.}  Fix $x<t_\infty$.  For every $k$ with $t_k\in(x,t_\infty)$
the path $P_{B,e_n}$ passes through $[B\cap H_{t_k}]=[b_k]=s_k$.  Since each $\sigma_m$
occurs as $s_k$ for infinitely many $k$, it occurs for some such $k$ with $t_k>x$;
hence $\{\sigma_m\}\subseteq\overline{P_{B,e_n}((x,t_\infty))}$, and taking closures,
$\overline{P_{B,e_n}((x,t_\infty))}=\A_{n-1}$.  As this holds for every $x<t_\infty$,
\[
  T(B,e_n)=\bigcap_{x<t_\infty}\overline{P_{B,e_n}((x,t_\infty))}=\A_{n-1}. \qedhere
\]
\end{proof}

\begin{remark}
The last support height $t_\infty$ is not prescribed but determined by the increments
$\tau_k$, and the fast shrinkage $d_k\le\tfrac18d_{k-1}$ enters only through
Lemma~\ref{lem:majorant}.  The centring $0\in\intr b_k$ keeps $0\in\intr N_{k}$ at every
stage, so the recursion does not stall.
\end{remark}

\subsection*{Grafting}
To distribute nearly universal points over a boundary we graft shrunken copies of the
cusp of Theorem~\ref{thm:singlenu} onto a body, one at a time, each new graft far
smaller than the working scale of every cusp already present.  The stability estimate
below shows that the sections at each established cusp move by an amount that is
$O(s)$ at depth $s$, hence negligible against a section width $w(s)\gg s$.  The cusp geometry itself supplies the room: a section
that shrinks with vertical tangent dwarfs, at small depth, the shadow cast through it
by anything planted below.

Let $K\subseteq\overline B(0,2)$ be a body whose last support set in a direction $\rho$
is a single (exposed) point $p$.  Define the \emph{depth} of $x$ as
$s(x)=\langle p-x,\rho\rangle$, so $s\ge0$ on $K$ with equality only at $p$; write
$K_s=K\cap\{s(\cdot)=s\}$ and $w(s)=\diam K_s$.  Call the cusp at $p$ \emph{steep} if
$s/w(s)\to0$ as $s\to0^+$.  The construction of Theorem~\ref{thm:singlenu} produces
steep cusps: the recursion leaves the level gaps $\tau_k$ free, and taking them small
against the diameters $d_k$ gives $s/w(s)\to0$ along the designed levels, hence
between them as well: the width is concave and vanishes at $s=0$, so $s/w(s)$ is
nondecreasing, and its smallness along levels decreasing to $0$ controls it
everywhere below.

\begin{lemma}[Grafting]\label{lem:graft}
Let $K,\rho,p$ be as above, with depth range $[0,s_K]$, and let $G$ be a compact
set within $\overline B(0,2)$ all of whose points have depth $\ge\eta>0$.  Set
$B_0=\overline{\conv(K\cup G)}$ and, for $j\ge1$,
$B_j=\overline{\conv(B_{j-1}\cup F_j)}$, where
$F_j\subseteq\overline B(y_j,R_j)\cap\overline B(0,2)$ is compact, seated at a
point $y_j\in B_{j-1}$ of depth $s_j>0$ (seats on earlier grafts are allowed),
subject to the rule
\[
  R_j\ \le\ \tfrac14\,\eps_j\min\big(s_j^{\,2},\ s_j\big),\qquad
  \textstyle\sum_j\eps_j\le1 .
  \tag{$\dagger$}
\]
Put $B=\overline{\bigcup_jB_j}$ and $\bar s=\tfrac16\min(\eta,s_K)$.  Then, with
$s'(x)=\langle p-x,\rho\rangle$ the depth function of $B$:
\begin{enumerate}
\item[(a)] $p$ is the unique point of $B$ at depth $0$ and remains exposed in
direction $\rho$;
\item[(b)] for every $s\in(0,\bar s\,]$,
$\ \dH\big(B\cap\{s'=s\},\,K_s\big)\ \le\ C\,s$, where
$C=17+\dfrac{10}{\eta}+\dfrac{16}{s_K}$;
\item[(c)] if the cusp at $p$ is steep, the a-shape paths of $B$ and of $K$ in
direction $\rho$ have the same limit points as $s\to0$; in particular
$T(B,\rho)=T(K,\rho)$, and the cusp of $B$ at $p$ is again steep.
\end{enumerate}
\end{lemma}

\begin{proof}
(a)  By $(\dagger)$ every point of $F_j$ has depth $\ge s_j-R_j\ge\tfrac34s_j>0$, and every
point of $G$ has depth $\ge\eta>0$; so $p$ is the unique depth-$0$ point of the
generating set, hence of $B$, and the unique maximiser of $\langle\cdot,\rho\rangle$.

(b)  Fix a transverse unit vector $u'$ and, for a body $C$ with $p\in C$, let
$\psi_C(\sigma)=\sup\{\langle x-p,u'\rangle: x\in C,\ s'(x)=\sigma\}$, the section
support of $C$ in direction $u'$ measured from $p$, defined on the depth range of
$C$.  Each $\psi_C$ is concave (Lemma~\ref{lem:convexbody}), with $\psi_C(0)=0$ by
the uniqueness in (a) applied to the stages already built, and $|\psi_C|\le4=:\Delta$,
since $K$, $G$ and every $F_j$ lie in $\overline B(0,2)$ and hence so does every
$B_j$.  Write $\psi=\psi_{B_{j-1}}$; its domain contains $[0,s_K]$, and for
$0<\sigma\le s_K/2$ concavity gives the two-sided slope bound
\[
  -\frac{4\Delta}{s_K}\ \le\ \psi'(\sigma)\ \le\ \frac{\psi(\sigma)}{\sigma}
  \ \le\ \frac{\Delta}{\sigma}\,,
\]
the lower estimate by comparison with a point of the graph at depth $s_K$, the
upper by $\psi(0)=0$.

We compare successive stages at a fixed depth $s\in(0,\bar s\,]$.  By the envelope
identity (Remark~\ref{rem:envcont}), $\psi_{B_j}$ is the least concave majorant of
the graph of $\psi$ together with the points $(\sigma,v)$ contributed by $F_j$;
by $(\dagger)$ these satisfy $\sigma\in[\tfrac34s_j,\tfrac54s_j]$ and
$v\le\psi(s_j)+R_j$, the latter because the seat $y_j$ lies in $B_{j-1}$ at depth
$s_j$.  Chords joining two contributed points are majorised by treating each end
separately, since the weights on the ends sum to $1$.

\emph{Shallow grafts $(s_j\le s_K/3)$.}  Every contributed abscissa satisfies
$\sigma\le\tfrac54s_j\le s_K/2$, and $[\tfrac34s_j,\tfrac54s_j]$ lies where the
slope bound applies, so
$|\psi(s_j)-\psi(\sigma)|\le\max\big(\tfrac{4\Delta}{3s_j},\tfrac{4\Delta}{s_K}\big)R_j$
and hence, by $(\dagger)$ (using $R_j\le\tfrac14\eps_js_j^{\,2}$ in the middle term
and $s_j\le s_K/3$ in the last),
\[
  v\ \le\ \psi(\sigma)+R_j+\frac{4\Delta}{3s_j}R_j+\frac{4\Delta}{s_K}R_j
  \ \le\ \psi(\sigma)+(\Delta+1)\,\eps_j s_j .
\]
The concave function $\Theta(x)=\psi(x)+\tfrac43(\Delta+1)\eps_j\,x$ therefore
dominates the graph and every contributed point---at $x=\sigma\ge\tfrac34s_j$ the
added linear term is at least $(\Delta+1)\eps_js_j$---so $\psi_{B_j}\le\Theta$ and
the increment at $s$ is at most $\tfrac43(\Delta+1)\eps_j\,s\le7\eps_js$.

\emph{Deep grafts $(s_j>s_K/3)$.}  Here $s\le\bar s\le s_K/6<s_j/2$, and every
contributed abscissa satisfies $\sigma\ge\tfrac34s_j>s$.  A chord from a
contributed point $(\sigma,v)$ to a graph point $(\sigma_0,\psi(\sigma_0))$ with
$\sigma_0<s$ has value
$\psi(\sigma_0)+(s-\sigma_0)\tfrac{v-\psi(\sigma_0)}{\sigma-\sigma_0}$ at $s$,
while concavity gives
$\psi(s)\ge\psi(\sigma_0)+(s-\sigma_0)\tfrac{a}{s_j-\sigma_0}$ with
$a=\psi(s_j)-\psi(\sigma_0)$.  Using $v-\psi(\sigma_0)\le a+R_j$, $|a|\le2\Delta$,
$|\sigma-s_j|\le R_j$,
$\sigma-\sigma_0\ge\sigma-s\ge\tfrac34s_j-\tfrac12s_j=\tfrac14s_j$ and
$s_j-\sigma_0\ge s_j-s\ge\tfrac12s_j$, the excess of the chord over $\psi(s)$ is
at most
\[
  (s-\sigma_0)\Big[\frac{R_j}{\sigma-\sigma_0}
  +|a|\,\frac{|\,s_j-\sigma\,|}{(\sigma-\sigma_0)(s_j-\sigma_0)}\Big]
  \ \le\ s\Big[\frac{4R_j}{s_j}+\frac{16\Delta R_j}{s_j^{\,2}}\Big]
  \ \le\ (4\Delta+1)\,\eps_j\,s\ =\ 17\,\eps_js ,
\]
by $(\dagger)$.

In either case $\psi_{B_j}(s)-\psi_{B_{j-1}}(s)\le17\,\eps_js$, and the increments
telescope: with $\sum_j\eps_j\le1$,
$\ \psi_B(s)-\psi_{B_0}(s)\le17\,s$ (passing to the closure of the increasing union
changes no section at interior depths, as in Remark~\ref{rem:envcont}).

\emph{The host.}  The same comparison applied once to
$B_0=\overline{\conv(K\cup G)}$ has contributed points from $G$ at depths
$\sigma_G\ge\eta\ge6s$ with values $v_G\le\Delta$.  A chord from
$(\sigma_G,v_G)$ to $(\sigma_0,\psi_K(\sigma_0))$, $\sigma_0<s$, carries weight
$(s-\sigma_0)/(\sigma_G-\sigma_0)\le s/(\eta-s)\le\tfrac65\,s/\eta$ on the deep end,
so its excess over $\psi_K(s)$ is at most
$\tfrac65(s/\eta)\,2\Delta+\big(\psi_K(\sigma_0)-\psi_K(s)\big)
\le\tfrac{12\Delta}{5}\,\tfrac{s}{\eta}+\tfrac{4\Delta}{s_K}\,s$,
the second term by the slope bound on $[0,s_K/2]$.  With $\Delta=4$ this is at most
$(10/\eta+16/s_K)\,s$, so $\psi_{B_0}(s)-\psi_K(s)\le(10/\eta+16/s_K)\,s$.

Altogether, for every $u'$ and every $s\in(0,\bar s\,]$,
\[
  0\ \le\ h_{B\cap\{s'=s\}}(u')-h_{K_s}(u')
  \ \le\ \Big(17+\frac{10}{\eta}+\frac{16}{s_K}\Big)s ,
\]
and since $B\cap\{s'=s\}\supseteq K_s$, the Hausdorff distance between the sections
is the supremum over $u'$ of the left-hand difference.

(c)  By (b), $\dH\big(B\cap\{s'=s\},K_s\big)\le Cs=o\big(w(s)\big)$ as $s\to0$, by
steepness.  Renormalising by diameters that differ by $O(s)$ as well, the a-shapes
of the two sections converge in $\dA$ as $s\to0$, so the two paths have identical
limit sets and $T(B,\rho)=T(K,\rho)$.  Finally $w_B(s)\ge w(s)$ gives
$s/w_B(s)\to0$.
\end{proof}

\begin{remark}\label{rem:shadowrule}
The rule $(\dagger)$ is sharp in kind.  Placing a second cusp apex on the unit sphere
at angle $\theta$ from $\rho$ seats it, relative to the first apex, at depth
$\asymp\theta^2$ with transverse extent $\asymp\theta$: the ratio (extent)/(depth)
diverges, $(\dagger)$ fails at every scale, and indeed the chord from such a point
spreads the sections at depth $s$ by $\asymp\sqrt{s}$, which swamps any steep cusp.
This is why apexes of equal height cannot cluster and every grafting scheme must be
hierarchical: each new cusp shallow as well as small compared with the cusps it lands
among.
\end{remark}

\subsection*{Nearly universal points dense in the boundary}

\begin{theorem}\label{thm:densenu}
For every $n\ge3$ there is a body $B\subseteq\Rn$ whose nearly universal points are
dense in $\partial B$.  More generally, given any sequence
$\Gamma_1,\Gamma_2,\dots\subseteq\A_{n-1}$ of sets realisable as tails, the grafting
sites may be chosen dense in $\partial B$ with $T(B,\rho_j)=\Gamma_j$ at the $j$-th
site, where $\rho_j$ is the corresponding normal direction.
\end{theorem}

\begin{proof}
Set $B_0=\overline B(0,1)$ and construct $B_j=\overline{\conv(B_{j-1}\cup F_j)}$, where
$F_j$ is a gadget produced by Theorem~\ref{thm:singlenu}---built steep, with
designed a-shapes forming a sequence whose limit set is $\Gamma_j$ and whose
consecutive steps tend to $0$ in $\dA$, so that the interpolating Minkowski
segments accumulate only on $\Gamma_j$ (Lemma~\ref{lem:mlshort}); such a sequence
exists because each $\Gamma_j$ is a continuum (Theorem~\ref{thm:continuum}):
concatenate $\eps$-chains within $\Gamma_j$ at scales $\eps\downarrow0$.  (For
$\Gamma_j=\A_{n-1}$ the step condition is unnecessary, and a dense sequence with
each term repeated infinitely often suffices.)  The gadget is
scaled to lie in $\overline B(x_j,R_j)$ and reflected so that its apex is
$a_j=x_j+\eta_jR_j'\,u_j$ for a seat $x_j\in\partial B_{j-1}$, an outer unit normal
$u_j$ of $B_{j-1}$ at $x_j$, and a protrusion $0<\eta_jR_j'\le R_j$.  The data
$(x_j,R_j)$ are chosen subject to:
\begin{enumerate}
\item[(i)] $R_j\le\tfrac14\,2^{-j}\min\big(s_{j,i}^{\,2},\,s_{j,i}\big)$ for every $i<j$, where
$s_{j,i}=\langle a_i-x_j,\,u_i\rangle$ is the depth of the seat below the $i$-th apex;
\item[(ii)] $R_j\le2^{-j}$;
\item[(iii)] the seats follow a schedule, described below, making them dense in
$\partial B$.
\end{enumerate}
Each $s_{j,i}$ is positive: $a_i$ is, by induction and Lemma~\ref{lem:graft}(a), the
unique maximiser of $\langle\cdot,u_i\rangle$ on $B_{j-1}$, and the schedule takes
$x_j\ne a_i$.  So (i) imposes finitely many positive upper bounds and can always be
met.

\emph{Each apex retains its designed tail.}  Fix $j$ and apply Lemma~\ref{lem:graft}
with $K=$ the gadget $F_j$, $\rho=u_j$, $p=a_j$, taking as the deep set $G=B_{j-1}$
(every point of $B_{j-1}$ has depth
$\ge\eta_jR_j'>0$ below $a_j$, because $x_j$ maximises $\langle\cdot,u_j\rangle$ on
$B_{j-1}$) and as the grafted sequence the later gadgets $F_{j'}$, $j'>j$, whose seats
have depths $s_{j',j}>0$ and radii obeying $(\dagger)$ with $\eps_{j'}=2^{-j'}$,
by (i).  Parts (b) and (c) give:
$a_j$ stays exposed in direction $u_j$ under all later grafts, the cusp at $a_j$ stays
steep, and the a-shape path of the final body $B=\overline{\bigcup_jB_j}$ in direction
$u_j$ has the same limit set as the designed path of $F_j$; that is,
$T(B,u_j)=\Gamma_j$, and for $\Gamma_j=\A_{n-1}$ the point $a_j$ is nearly universal.
(Passing from the increasing union to its closure changes no section at positive
depth, by the interior--closure identity used in Theorem~\ref{thm:singlenu}.)

\emph{The schedule.}  By (ii), $\dH(B_{j},B)\le\sum_{j'>j}2^{-j'}\to0$, and since all
bodies contain $\overline B(0,1)$, the boundaries converge in Hausdorff distance as
well (boundaries of convex bodies with a common interior ball converge together
with the bodies; see Schneider, \emph{Convex Bodies}, \S1.8).  Enumerate, for each $k$, a countable dense subset $Z_k$ of $\partial B_k$, and
process the pairs $(k,z)$, $z\in Z_k$, in a single diagonal sequence; at the stage
handling $(k,z)$, take the seat $x_j$ to be a point of $\partial B_{j-1}$ within
$2^{-j}+\dH(\partial B_k,\partial B_{j-1})$ of $z$, avoiding the countably many earlier
apexes.  Every point of $\partial B$ is a limit of points of $\partial B_k$ for large
$k$, hence a limit of scheduled targets, hence a limit of seats, hence---by (ii)---a
limit of apexes $a_j$.  The apexes are therefore dense in $\partial B$, and each
carries its designed tail.
\end{proof}

\begin{remark}
Theorem~\ref{thm:densenu} contains the finite and countable simultaneous statements:
prescribed tails $\Gamma_1,\dots,\Gamma_m$ (or a countable sequence) in prescribed
directions are obtained by seating the first $m$ grafts at the support points of those
directions and letting the schedule fill in the rest, or stop.  It also shows that the
set of non-boring points of a body can be dense in its boundary, sharpening the
existence statement of Theorem~\ref{thm:singlenu}.
\end{remark}

The construction controls more than the tails: it controls the paths.  Let
$\mathcal P$ denote the space of continuous maps $(0,1]\to\A_{n-1}$ with the topology
of uniform convergence on compact subsets of $(0,1]$; since $\A_{n-1}$ is compact
metric, $\mathcal P$ is separable and completely metrisable.  For a boundary point $p$
of $B$ exposed in direction $\rho$, write $Q_{B,p}\colon(0,\delta]\to\A_{n-1}$ for the
a-shape path of the sections of $B$ at depth $s$ below $p$.

\begin{corollary}\label{cor:pathdense}
The body $B$ of Theorem~\ref{thm:densenu} may be constructed so that, for a prescribed
countable family $\{P_m\}\subseteq\mathcal P$ and $\eps_m\to0$, each $P_m$ is shadowed
at some site: there are $j=j(m)$ and an increasing homeomorphism
$\varphi\colon(0,\delta_j]\to(0,1]$ with
$\dA\big(Q_{B,a_j}(s),\,P_m(\varphi(s))\big)\le\eps_m$ for all $s\in(0,\delta_j]$.
Taking $\{P_m\}$ dense, the paths at the grafting sites of a single body are dense in
$\mathcal P$, up to increasing reparametrisation.
\end{corollary}

\begin{proof}
Design the gadget at stage $j$ against a pair $(P_m,\eps_m)$, each pair recurring
infinitely often in the schedule.  Between consecutive designed levels the section
supports of the gadget are linear in the height (Lemma~\ref{lem:envelope}), so its
a-shape path is the piecewise-Minkowski-linear path through the designed a-shapes, as
in Theorem~\ref{thm:sufficiency}.  Choose sampling parameters $1=u_1>u_2>\cdots\to0$ by covering each dyadic block
$[2^{-l-1},2^{-l}]$ with a finite net on which the oscillation of $P_m$ between
consecutive points is below $\eps_m$; uniform continuity of $P_m$ on each block
supplies the nets, and their union accumulates only at $0$, so the samples indeed
decrease to $0$.  By Lemma~\ref{lem:mlshort} the Minkowski segment between the
designed shapes at consecutive samples then stays within a fixed multiple of
$\eps_m$ of $P_m$ on $[u_{k+1},u_k]$.  Take the $k$-th designed shape to be a
nondegenerate shape within $\eps_m$ of $P_m(u_k)$, possible since
$\Ac_{n-1}$ is dense (Proposition~\ref{prop:open-connected}).  The recursion of
Theorem~\ref{thm:singlenu} constrains only the sizes $d_k$ and gaps $\tau_k$ from
above, never the shapes, so the sampling may be refined freely; let $\varphi$ send the
$k$-th designed depth to $u_k$ and interpolate linearly.  Taking the gadget steep
enough that $\sup_{s\le\delta_j}\,C_j\,s/w(s)\le\eps_m$, with $C_j$ the constant of
Lemma~\ref{lem:graft}(b) for the $j$-th site and $\delta_j$ within its range, the
perturbation of the realised path by
the host and all later grafts is a further $\eps_m$ uniformly on $(0,\delta_j]$.  This
proves the shadowing statement, with a fixed multiple of $\eps_m$ in place of
$\eps_m$; density follows,
since uniform closeness to $P_m$ on all of $(0,1]$ is stronger than closeness in
$\mathcal P$.
\end{proof}

\begin{remark}
The reparametrisation cannot be dispensed with: the level gaps $\tau_k$ of
Theorem~\ref{thm:singlenu} are forced to shrink rapidly, so the realised path
traverses its designed shapes on a schedule of its own.  Up to that schedule, a single
body exhibits, at a dense set of boundary points, a dense set of ways of approaching
the degenerate contact.
\end{remark}

\section{Optimisation and higher codimension}\label{sec:optim}

The framework suggests several optimisation questions.  Suppose a path $P$ is realised
by a body $B$, with $\rho$ vertical.  Normalising to base diameter $1$ and height $1$,
one asks for the body of \emph{minimum} volume realising $P$.  (The supremum of
volume is the base area, approached by building a cylinder and squashing the original
$B$, so only the minimum poses a question.)

\begin{question}\label{q:minvol}
Among bodies realising a given cross-section path $P$ with base diameter and height
$1$, is the infimum of volume attained, and what does a minimiser look like near the
last support point?
\end{question}

There is also a higher-codimension variant.  Given a body $J\subseteq\mathbb R^i$
(the projection) and a map $Q\colon J\to\A_{n-i}$, one seeks a body $B\subseteq\Rn$
projecting onto $J$ such that, under a suitable parameterisation, the fibres of the
projection realise $Q$.  Certain cases follow by combining
Theorem~\ref{thm:necessity} with an inductive application of the single-point
construction; a full analysis is left to future work.

\section{Open problems}\label{sec:open}

\begin{question}\label{q:exactchar}
Characterise the exactly realisable paths intrinsically.  Local rectifiability is
necessary (Theorem~\ref{thm:necessity}) but not sufficient
(Theorem~\ref{thm:norealize}); beyond it, all mixed-width ratios must admit a common
reparameterisation after which each has derivative of locally bounded variation
(Lemma~\ref{lem:mixedwidth}), and the full requirement is that a single scale,
translation, and reparameterisation render every section support concave at once.
For box paths the half-power dichotomy of Theorem~\ref{thm:norealize} and
Proposition~\ref{prop:hierzig} suggests a complete answer may be within reach;
already open is whether the amplitude restriction in
Proposition~\ref{prop:hierzig} can be removed.  In
general: do the mixed-width conditions suffice, or do functionals beyond the widths
impose independent constraints?
\end{question}

\begin{question}
Does every body possess a boring point?  Is there a body all of whose boundary points
are non-boring?
\end{question}

\begin{question}
What are the possible topological and Borel types of the set of non-boring points on
the boundary of a body?  Theorem~\ref{thm:continuum} constrains the individual tails,
and Theorem~\ref{thm:densenu} shows the nearly universal points can be dense; can they
have full measure, or be residual in $\partial B$, or is the set of non-boring points
always small in some quantitative sense?
\end{question}

\begin{question}\label{q:typical}
What are the non-boring points of a random convex body?  The common probabilistic
models give degenerate answers: the convex hull of finitely many random points is a
polytope, boring in every direction (Proposition~\ref{prop:poly} at vertices;
elsewhere the path terminates), and smooth
models of everywhere positive curvature are boring throughout
(Proposition~\ref{prop:smooth}).  The question is alive for typical bodies in the
sense of Baire category, where most bodies are smooth and strictly convex yet nowhere
twice differentiable, with extreme curvature behaviour at most boundary points---the
raw material of Proposition~\ref{prop:degenerate}, though our constructions need it in
a quantitatively fast form.  Is the set of non-boring points of a typical body empty,
nonempty, dense?  Do typical bodies possess nearly universal points?
\end{question}

\begin{question}
Which continua in $\A_{n-1}$ occur as tails?  Every tail is a continuum
(Theorem~\ref{thm:continuum}) and $\A_{n-1}$ itself occurs (Theorem~\ref{thm:sufficiency});
is every continuum a tail?
\end{question}


\begin{thebibliography}{9}

\bibitem{Lean4}
L.~de~Moura and S.~Ullrich,
\emph{The Lean 4 theorem prover and programming language},
in: Automated Deduction -- CADE 28,
Lecture Notes in Computer Science 12699, Springer, 2021, 625--635.

\bibitem{mathlib}
The mathlib Community,
\emph{The Lean mathematical library},
in: Proceedings of the 9th ACM SIGPLAN International Conference on
Certified Programs and Proofs (CPP 2020), ACM, 2020, 367--381.

\bibitem{Schneider}
R.~Schneider,
\emph{Convex Bodies: The Brunn--Minkowski Theory},
second expanded edition,
Encyclopedia of Mathematics and its Applications 151,
Cambridge University Press, 2014.

\end{thebibliography}
\end{document}